\documentclass[a4paper,11pt]{amsart}

\usepackage{amssymb}
\usepackage[utf8]{inputenc}
\usepackage[T1]{fontenc}
\usepackage{mathptmx}

\makeatletter
\@namedef{subjclassname@2010}{%
  \textup{2010} Mathematics Subject Classification}
\makeatother

\newtheorem{theorem}{Theorem}[section]
\newtheorem{proposition}[theorem]{Proposition}
\newtheorem{lemma}[theorem]{Lemma}
\newtheorem{corollary}[theorem]{Corollary}
\theoremstyle{definition}

\numberwithin{equation}{section}

\DeclareMathOperator{\soporte}{supp}
\DeclareMathOperator{\traza}{Tr}

\begin{document}

\title[Orthogonally additive polynomials]
{Orthogonally additive polynomials on non-commutative $L^p$-spaces}

\author{J. Alaminos}
\address{Departamento de An\' alisis
Matem\' atico\\ Fa\-cul\-tad de Ciencias\\ Universidad de Granada\\
18071 Granada, Spain} 
\email{alaminos@ugr.es}
\author{M. L. C. Godoy}
\address{Departamento de An\' alisis
Matem\' atico\\ Fa\-cul\-tad de Ciencias\\ Universidad de Granada\\
18071 Granada, Spain} 
\email{marisa23@correo.ugr.es}
\author{A.\,R. Villena}
\address{Departamento de An\' alisis
Matem\' atico\\ Fa\-cul\-tad de
 Ciencias\\ Universidad de Granada\\
18071 Granada, Spain} 
\email{avillena@ugr.es}

\date{}

\begin{abstract}
Let $\mathcal{M}$ be a von Neumann algebra with a normal semifinite faithful trace $\tau$.
We prove that every continuous $m$-homogeneous polynomial $P$ from $L^p(\mathcal{M},\tau)$,
with $0<p<\infty$, into each topological linear space $X$ 
with the property that $P(x+y)=P(x)+P(y)$ whenever $x$ and $y$ are mutually orthogonal positive elements of $L^p(\mathcal{M},\tau)$
can be represented  in the form $P(x)=\Phi(x^m)$ 
$(x\in L^p(\mathcal{M},\tau))$ for some continuous linear map 
$\Phi\colon L^{p/m}(\mathcal{M},\tau)\to X$.
\end{abstract}


\thanks{
The authors were supported by MINECO grant MTM2015--65020--P.
The first and the third named authors were supported by Junta de Andaluc\'{\i}a grant FQM--185. 
The second named author was supported by Beca de Iniciaci\'on a la Investigaci\'on de la Universidad de Granada.
}

\maketitle

\section{Introduction}

In \cite{S}, the author succeeded in providing a useful representation 
of the orthogonally additive homogeneous 
polynomials on the spaces $L^p([0,1])$ and $\ell^p$ with $1\le p<\infty$. 
In~\cite{PV} (see also~\cite{CLZ}), the authors obtained a similar 
representation for the space $C(K)$, for a compact Hausdorff space $K$. 
These results were generalized to Banach lattices~\cite{BLL} and Riesz 
spaces ~\cite{ILL}. Further, the problem of representing the orthogonally additive homogeneous polynomials has 
been also considered in the context of Banach function algebras~\cite{AEV,V} and non-commutative 
Banach algebras~\cite{AEGV,AGV,P}. Notably, \cite{P} can be thought of as the natural non-commutative analogue
of the representation of orthogonally additive polynomials on $C(K)$-spaces, and
the purpose to this paper is to extend the results of~\cite{S} on the representation of orthogonally 
additive homogeneous polynomials on $L^p$-spaces to the non-commutative $L^p$-spaces.

The non-commutative $L^p$-spaces that we consider are those associated with a von Neumann algebra
$\mathcal{M}$ equipped with a normal semifinite faithful trace $\tau$.
From now on, $S(\mathcal{M},\tau)$ stands for the linear span of the positive elements $x$ of $\mathcal{M}$ such that
$\tau\bigl(\soporte(x)\bigr)<\infty$; here $\soporte (x)$ stands for the support of $x$.
Then $S(\mathcal{M},\tau)$ is a $\ast$-subalgebra of $\mathcal{M}$ with the property that
$\vert x\vert^p\in S(\mathcal{M},\tau)$ for each $x\in S(\mathcal{M},\tau)$ and each $0<p<\infty$.
For $0<p<\infty$, we define $\Vert\cdot\Vert_p\colon S(\mathcal{M},\tau)\to\mathbb{R}$ by
$\Vert x\Vert_p=\tau(\vert x\vert^p)^{1/p}$ $(x\in S(\mathcal{M},\tau))$.
Then $\Vert\cdot\Vert_p$ is a norm or a $p$-norm according to $1\le p<\infty$ or $0<p<1$, and 
the space $L^p(\mathcal{M},\tau)$ can be defined as the completion of $S(\mathcal{M},\tau)$
with respect to $\Vert\cdot\Vert_p$.
Nevertheless, for our purposes here, it is important to realize the elements of $L^p(\mathcal{M},\tau)$ as
measurable operators. Specifically, the set $L^0(\mathcal{M},\tau)$ of measurable closed densely defined operators 
affiliated to $\mathcal{M}$ is a topological $*$-algebra with respect to the strong sum, the strong product,
the adjoint operation, and the topology of the convergence in measure.
The algebra $\mathcal{M}$ is a dense $*$-subalgebra of $L^0(\mathcal{M},\tau)$, 
the trace $\tau$ extends to the positive cone of $L^0(\mathcal{M},\tau)$ in a natural way, and 
we can define
\begin{gather*}
\Vert x\Vert_p=\tau\bigl(\vert x\vert^p\bigr)^{1/p}\quad (x\in L^0(\mathcal{M},\tau)), \\
L^p(\mathcal{M},\tau)=\bigl\{x\in L^0(\mathcal{M},\tau) : \Vert x\Vert_p<\infty\bigr\}.
\end{gather*}
Also we set $L^\infty(\mathcal{M},\tau)=\mathcal{M}$ (with $\Vert\cdot\Vert_\infty:=\Vert\cdot\Vert$, the operator norm).
Operators $x,y\in L^0(\mathcal{M},\tau)$ are mutually orthogonal, written $x\perp y$,
if $xy^*=y^*x=0$. This condition is equivalent to requiring that $x$ and $y$ have mutually orthogonal left, and right,
supports. Further, for $x,y\in L^p(\mathcal{M},\tau)$ with $0<p<\infty$, 
the condition $x\perp y$ implies that $\Vert x+y\Vert_p^p=\Vert x\Vert_p^p+\Vert y\Vert_p^p$, 
and conversely, if  $\Vert x\pm y\Vert_p^p=\Vert x\Vert_p^p+\Vert y\Vert_p^p$ and $p\ne 2$, then $x\perp y$ 
(see \cite[Fact 1.3]{R}).
The orthogonal additivity considered in \cite{S} for the spaces $L^p([0,1])$ and $\ell^p$ can of course equally well
be considered for the space $L^p(\mathcal{M},\tau)$.
Let $P$ be a map from $L^p(\mathcal{M},\tau)$ into a linear space $X$. Then $P$ is:
\begin{enumerate}
\item[(i)]
\emph{orthogonally additive} on a subset $\mathcal{S}$ of $L^p(\mathcal{M},\tau)$ if
\begin{equation*}
x,y \in \mathcal{S}, \, x\perp y=0 \ \Rightarrow \ P(x+y)=P(x)+P(y);
\end{equation*}
\item[(ii)]
an \emph{$m$-homogeneous polynomial} if
there exists an $m$-linear map $\varphi$ from $L^p(\mathcal{M},\tau)^m$ into $X$ such that
\begin{equation*}
P(x)=\varphi \left( x,\dotsc,x \right)\quad(x\in L^p(\mathcal{M},\tau)).
\end{equation*} 
Here and subsequently, $m\in\mathbb{N}$ is fixed with $m\ge 2$ and 
the superscript $m$ stands for the $m$-fold Cartesian product. 
Such a map is unique if it is required to be symmetric. 
Further, in the case where $X$ is a topological linear space, the
polynomial $P$ is continuous if and only if the symmetric $m$-linear
map $\varphi$ associated with $P$ is continuous.
\end{enumerate}
Given a continuous linear map $\Phi\colon L^{p/m}(\mathcal{M},\tau)\to X$, 
where $X$ is an arbitrary topological linear space, the map $P_\Phi\colon L^p(\mathcal{M},\tau)\to X$
defined by 
\begin{equation*}
P_\Phi(x)=\Phi(x^m)\quad(x\in L^p(\mathcal{M},\tau))
\end{equation*}
is a natural example of a continuous $m$-homogeneous polynomial which is orthogonally additive on
$L^p(\mathcal{M},\tau)_{\textup{sa}}$ (Theorem~\ref{1242}), and we will prove that every continuous 
$m$-homogeneous polynomial which is orthogonally additive on $L^p(\mathcal{M},\tau)_{\textup{sa}}$
is actually of this special form (Theorem~\ref{1527}). Here and subsequently,
the subscripts ``sa'' and $+$ are used to denote the self-adjoint and the positive parts
of a given subset of $L^0(\mathcal{M},\tau)$, respectively.

We require a few remarks about the setting of our present work.
Throughout the paper we are concerned with $m$-homogeneous polynomials on the space $L^p(\mathcal{M},\tau)$ 
with $0<p$, and thus one might wish to consider polynomials with values in the space $L^q(\mathcal{M},\tau)$,
especially with $q\le p$.
Further, in the case case where $p/m<1$ and the von Neumann algebra $\mathcal{M}$ has no minimal projections, 
there are no non-zero continuous linear functionals on $L^{p/m}(\mathcal{M},\tau)$;
since one should like to have non-trivial ``orthogonally additive'' polynomials on $L^p(\mathcal{M},\tau)$, 
some weakening of the normability must be allowed to the range space (see Corollary~\ref{1416}).
For these reasons, throughout the paper, $X$ will be a (complex and Hausdorff) topological linear space.
In the case where the von Neumann algebra $\mathcal{M}$ is commutative,
the prototypical polynomials $P_\Phi$ mentioned above are easily seen to be orthogonally additive on
the whole domain. 
In contrast, we will point out in Propositions~\ref{1525} and~\ref{1526}
that this is not the case for the von Neumann algebra
$\mathcal{B}(H)$ of all bounded operators on a Hilbert space $H$ whenever $\dim H\ge 2$.

We assume a basic knowledge of $C^*$-algebras and von Neumann algebras, 
tracial non-commutative $L^p$-spaces, and
polynomials on topological linear spaces.
For the relevant background material concerning these topics, see
\cite{B,DAV,Tak1}, \cite{Ne,Pi,Te}, and \cite{Di}, respectively.

\section{$C^*$-algebras and von Neumann algebras} 

Our approach to the problem of representing the orthogonally 
additive homogeneous polynomials on the non-commutative $L^p$-spaces 
relies on the representation of those polynomials on the von Neumann 
algebras. Although \cite{P} solves the problem of representing the 
orthogonally additive homogeneous polynomials from
a $C^*$-algebra into an arbitrary Banach space, the greater 
generality of our range space does not allow us
to apply the results of that paper. Thus our first aim is to 
extend~\cite{P} to polynomials with values in a topological linear space.

Recall that two elements $x$ and $y$ of a $C^*$-algebra $\mathcal{A}$ are mutually orthogonal
if $xy^*=y^*x=0$, in which case the identity $\Vert x+y\Vert=\max\{\Vert x\Vert,\Vert y\Vert\}$ holds.

Suppose that $\mathcal{A}$ is a linear space with an involution $\ast$.
Recall that for a linear functional $\Phi\colon \mathcal{A}\to \mathbb{C}$, 
the map $\Phi^*\colon \mathcal{A}\to\mathbb{C}$ defined by
$\Phi^*(x)=\overline{\Phi(x^{*})}$ $(x\in \mathcal{A})$ is a linear functional, and 
$\Phi$ is said to be \emph{hermitian} if $\Phi^*=\Phi$.
Similarly, for an $m$-homogeneous polynomial $P\colon \mathcal{A}\to \mathbb{C}$, 
the map $P^*\colon\mathcal{A}\to\mathbb{C}$
defined by $P^*(x)=\overline{P(x^*)}$ $(x\in\mathcal{A})$
is an $m$-homogeneous polynomial, and we call  $P$ \emph{hermitian}
if $P^*=P$. 

\begin{lemma}\label{l1}
Let $X$ and $Y$ be linear spaces, and let $P\colon X\to Y$ be an $m$-homogeneous polynomial.
Suppose that $P$ vanishes on a convex set $C\subset X$. 
Then $P$ vanishes on the linear span of $C$.
\end{lemma}

\begin{proof}
Set $x_1,x_2,x_3,x_4\in C$.
Let $\eta\colon Y\to\mathbb{C}$ be a linear functional, and 
define $f\colon\mathbb{C}^4\to\mathbb{C}$  by
\[
f(\alpha_1,\alpha_2,\alpha_3,\alpha_4)=
\eta\bigl(P(\alpha_1x_1+\alpha_2 x_2+\alpha_3x_3+\alpha x_4)\bigr)\quad
(\alpha_1,\alpha_2,\alpha_3,\alpha_4\in\mathbb{C}).
\]
Then $f$ is a complex polynomial function in four complex variables that vanishes on the
set
\[
\bigl\{(\rho_1,\rho_2,\rho_3,\rho_4)\in\mathbb{R}^4 :
0\le\rho_1,\rho_2,\rho_3,\rho_4, \ \rho_1+\rho_2+\rho_3+\rho_4=1\bigr\}.
\]
This implies that $f$ is identically equal to $0$ on $\mathbb{C}^4$, and, in particular, 
\[
\eta\bigl(P(\rho_1x_1-\rho_2x_2+i\rho_3x_3-i\rho_4x_4)\bigr)=f(\rho_1,-\rho_2,i\rho_3,-i\rho_4)=0
\]
for all $\rho_1,\rho_2,\rho_3,\rho_4\ge 0$.
Since this identity holds for each linear functional $\eta$, it may be concluded that 
$P(\rho_1x_1-\rho_2x_2+i\rho_3x_3-i\rho_4x_4)=0$ for all $\rho_1,\rho_2,\rho_3,\rho_4\ge 0$.
Thus $P$ vanishes on the set 
\[
\bigl\{
\rho_1x_1-\rho_2x_2+i\rho_3x_3-i\rho_4x_4 :
\rho_j\ge 0, \ x_j\in C \ (j=1,2,3,4)
\bigr\},
\]
which is exactly the linear span of the set $C$.
\end{proof}

\begin{theorem}\label{1522}
Let $\mathcal{A}$ be a $C^*$-algebra,
let $X$ be a topological linear space, and
let $\Phi\colon \mathcal{A}\to X$ be a continuous linear map.
Then:
\begin{enumerate}
\item[(i)] 
the map $P_\Phi\colon\mathcal{A}\to X$ defined by
$P_\Phi(x)=\Phi(x^m)$ $(x\in\mathcal{A})$ is a continuous $m$-homogeneous polynomial which
is orthogonally additive on $\mathcal{A}_{\textup{sa}}$;
\item[(ii)]
the polynomial $P_\Phi$ is uniquely specified by the map $\Phi$.
\end{enumerate}
Suppose, further, that $X$ is a $q$-normed space, $0<q\le 1$. Then:
\begin{enumerate}
\item[(iii)]
$2^{-1/q}\Vert\Phi\Vert\le\Vert P_\Phi\Vert\le\Vert\Phi\Vert$.
\end{enumerate}
Moreover, in the case where $X=\mathbb{C}$,
\begin{enumerate}
\item[(iv)]
the functional $\Phi$ is hermitian if and only if the polynomial $P_\Phi$ is hermitian, 
in which case $\Vert P_\Phi\Vert=\Vert\Phi\Vert$.
\end{enumerate}
\end{theorem}

\begin{proof}
(i)
It is clear that the map $P_\Phi$ is continuous and that $P_\Phi$ is the $m$-homogeneous polynomial associated with
the symmetric $m$-linear map $\varphi\colon \mathcal{A}^m\to X$ defined by
\[
\varphi(x_1,\ldots,x_m)= 
\frac{1}{m!}\sum_{\sigma \in \mathfrak{S}_m} 
\Phi\left(x_{\sigma(1)} \cdots x_{\sigma(m)} \right) \quad (x_1,\ldots,x_m\in\mathcal{A});
\]
here and subsequently, we write $\mathfrak{S}_m$ for the symmetric group of order $m$.

Suppose that $x,y\in\mathcal{A}_{\textup{sa}}$ are such that $x\perp y$. 
Then $xy=yx=0$, and so
$(x+y)^m=x^m+y^m$, which gives
\[
P_\Phi(x+y)=\Phi\bigl((x+y)^m\bigr)=
\Phi\bigl(x^m+y^m\bigr)=\Phi\bigl(x^m\bigr)+\Phi\bigl(y^m\bigr)=
P_\Phi(x)+P_\Phi(y).
\]

(ii)
Assume that $\Psi\colon\mathcal{A}\to X$ is a linear map with the property that $P_\Psi=P_\Phi$.
If $x\in\mathcal{A}_+$, then
\[
\Phi(x)=\Phi\bigl((x^{1/m})^{m}\bigr)=P(x^{1/m})=\Psi\bigl((x^{1/m})^{m}\bigr)=\Psi(x).
\]
By linearity we also get $\Psi(x) =  \Phi(x)$ for each $x \in \mathcal{A}$. 

(iii)
Next, assume that $X$ is a $q$-normed space.
For each $x\in\mathcal{A}$, we have
\[
\Vert P_\Phi(x)\Vert=\Vert\Phi(x^m)\Vert\le\Vert\Phi\Vert\Vert x^m\Vert\le\Vert\Phi\Vert\Vert x\Vert^m,
\] 
which implies that $\Vert P_\Phi\Vert\le\Vert\Phi\Vert$.
Now take $x\in\mathcal{A}$, and let $\omega\in\mathbb{C}$ with $\omega^m=-1$.
Then $x=\Re x+i\Im x$, where
\[
\Re x=\frac{1}{2}(x^*+x), \, 
\Im x=\frac{i}{2}(x^*-x)\in \mathcal{A}_{\textup{sa}},
\]
and, further,
$\Vert\Re x\Vert,\Vert\Im x\Vert\le\Vert x\Vert$.
Moreover, 
$\Re x=x_1-x_2$ and $\Im x=x_3-x_4$, 
where $x_1,x_2,x_3,x_4\in\mathcal{A}_+$, $x_1\perp x_2$, and $x_3\perp x_4$.
Since $x_1\perp x_2$ and $x_3\perp x_4$, it follows that
$x_1^{1/m}\perp x_2^{1/m}$ and $x_3^{1/m}\perp x_4^{1/m}$.
Consequently,
\begin{equation}\label{1925}
\begin{split}
\Vert\Re x\Vert
&=\max\bigl\{\Vert x_1\Vert,\Vert x_2\Vert\bigr\},\\
\Vert\Im x\Vert
&=\max\bigl\{\Vert x_3\Vert,\Vert x_4\Vert\bigr\},
\end{split}
\end{equation}
and 
\begin{equation}\label{1928}
\begin{split}
\bigl\Vert
x_1^{1/m}+\omega x_2^{1/m}
\bigr\Vert&=
\max\bigl\{
\bigl\Vert x_1^{1/m}\bigr\Vert,\bigr\Vert x_2^{1/m}
\bigr\Vert
\bigr\},\\
\bigl\Vert
x_3^{1/m}+\omega x_4^{1/m}
\bigr\Vert&=
\max\bigl\{
\bigl\Vert x_3^{1/m}\bigr\Vert,\bigr\Vert x_4^{1/m}
\bigr\Vert
\bigr\}.
\end{split}
\end{equation}
Since
\[
\bigl\Vert x_1^{1/m}\bigr\Vert=
\Vert x_1\Vert^{1/m}, \
\bigl\Vert x_2^{1/m}\bigr\Vert=
\Vert x_2\Vert^{1/m}, \
\bigl\Vert x_3^{1/m}\bigr\Vert=
\Vert x_3\Vert^{1/m}, \
\bigl\Vert x_4^{1/m}\bigr\Vert=
\Vert x_4\Vert^{1/m},
\]
it follows, from \eqref{1925} and \eqref{1928}, that
\begin{equation}\label{1929}
\begin{split}
\bigl\Vert x_1^{1/m}+\omega x_2^{1/m}\bigr\Vert^m
&=
\max\bigl\{\Vert x_1\Vert,\Vert x_2\Vert\bigr\}=\Vert\Re x\Vert,\\
\bigl\Vert x_3^{1/m}+\omega x_4^{1/m}\bigr\Vert^m
&=
\max\bigl\{\Vert x_3\Vert,\Vert x_4\Vert\bigr\}=\Vert\Im x\Vert.
\end{split}
\end{equation}
On the other hand, we have
\begin{equation*}
\bigl(x_1^{1/m}+\omega x_2^{1/m}\bigr)^m=x_1-x_2=\Re x,\quad
\bigl(x_3^{1/m}+\omega x_4^{1/m}\bigr)^m=x_3-x_4=\Im x,
\end{equation*}
and so
\begin{equation*}
\begin{split}
\Phi(x)
&=
\Phi(\Re x)+i\Phi(\Im x)=
\Phi\bigl(\bigl(x_1^{1/m}+\omega x_2^{1/m}\bigr)^m\bigr)+
i\Phi\bigl(\bigl(x_3^{1/m}+\omega x_4^{1/m}\bigr)^m\bigr)\\
&=
P_\Phi\bigl(x_1^{1/m}+\omega x_2^{1/m}\bigr)+
iP_\Phi\bigl(x_3^{1/m}+\omega x_4^{1/m}\bigr).
\end{split}
\end{equation*}
Hence, by \eqref{1929},
\begin{equation*}
\begin{split}
\Vert\Phi(x)\Vert^q
&\le
\bigl\Vert P_\Phi\bigl(x_1^{1/m}+\omega x_2^{1/m}\bigr)\bigr\Vert^q+
\bigl\Vert P_\Phi\bigl(x_3^{1/m}+\omega x_4^{1/m}\bigr)\bigr\Vert^q\\
&\le
\Vert P_\Phi\Vert^q\bigl\Vert x_1^{1/m}+\omega x_2^{1/m}\bigr\Vert^{mq}+
\Vert P_\Phi\Vert^q\bigl\Vert x_3^{1/m}+\omega x_4^{1/m}\bigr\Vert^{mq}\\
&=
\Vert P_\Phi\Vert^q\left(
\Vert\Re x\Vert^q+\Vert\Im x\Vert^q\right)\\
&\le
\Vert P_\Phi\Vert^q 2\Vert x\Vert^q.
\end{split}
\end{equation*}
This clearly forces $\Vert\Phi\Vert\le 2^{1/q}\Vert P_\Phi\Vert$, as claimed.

(iv)
It is straightforward to check that $P_\Phi^*=P_{\Phi^*}$. 
Consequently, if $\Phi$ is hermitian, then $P_\Phi^*=P_{\Phi^*}=P_\Phi$ so that
$P_\Phi$ is hermitian.
Conversely, if $P_\Phi$ is hermitian, then $P_{\Phi^*}=P_\Phi^*=P_\Phi$ and (ii) implies
that $\Phi^*=\Phi$.
Finally, assume that $\Phi$ is a hermitian functional.
For the calculation of $\Vert P_\Phi\Vert$ it suffices to check
that $\Vert\Phi\Vert\le\Vert P_\Phi\Vert$. For this purpose,
let $\varepsilon\in\mathbb{R}^+$, and choose $x\in\mathcal{A}$ such that
$\Vert x\Vert=1$ and 
$\Vert\Phi\Vert-\varepsilon<\vert\Phi(x)\vert$.
We take $\alpha\in\mathbb{C}$ with $\vert\alpha\vert=1$ and 
$\vert\Phi(x)\vert=\alpha\Phi(x)$, so that
\[
\Vert\Phi\Vert-\varepsilon<\vert\Phi(x)\vert=\Phi(\alpha x)=
\overline{\Phi(\alpha x)}=\Phi\bigl((\alpha x)^*\bigr).
\]
Note that $\Vert\Re(\alpha x)\Vert\le 1$ and $\Vert\Phi\Vert-\varepsilon<\Phi(\Re(\alpha x))$.
Now we consider the decomposition $\Re(\alpha x)=x_1-x_2$
with $x_1,x_2\in\mathcal{A}_+$ and $x_1\perp x_2$ and take $\omega\in\mathbb{C}$ with $\omega^m=-1$.
As in \eqref{1929}, we see that
$\bigl\Vert x_1^{1/m}+\omega x_2^{1/m}\bigr\Vert=\Vert\Re(\alpha x)\Vert^{1/m}\le 1$.
Moreover, we have
\[
P_\Phi\bigl( x_1^{1/m}+\omega x_2^{1/m}\bigr)=
\Phi\bigl(\bigl( x_1^{1/m}+\omega x_2^{1/m}\bigr)^m\bigr)=
\Phi(\Re(\alpha x)),
\]
which gives $\Vert\Phi\Vert-\varepsilon<\Vert P_\Phi\Vert$.
\end{proof}

\begin{lemma}\label{l12}
Let $\mathcal{A}$ be a $C^*$-algebra, 
let $\mathcal{R}$ be a $\ast$-subalgebra of $\mathcal{A}$, 
let $X$ be a topological linear space, and 
let $\Phi\colon\mathcal{R}\to X$ be a linear map. 
Suppose that the polynomial $P\colon\mathcal{R}\to X$ defined by
$P(x)=\Phi(x^m)$ $(x\in\mathcal{R})$ is continuous and that $\mathcal{R}$ satisfies the following conditions:
\begin{enumerate}
\item[(i)]
$\vert x\vert\in\mathcal{R}$ for each $x\in\mathcal{R}_{\textup{sa}}$;
\item[(ii)]
$x^{1/m}\in\mathcal{R}$ for each $x\in\mathcal{R}_+$.
\end{enumerate}
Then $\Phi$ is continuous.
\end{lemma}

\begin{proof}
Let $U$ be a neighbourhood of $0$ in $X$. Let $V$ be a balanced neighbourhood of $0$ in $X$ with
$V+V+V+V\subset U$. 
The set $P^{-1}(V)$ is a neighbourhood of $0$ in $\mathcal{R}$,
which implies that there exists $r\in\mathbb{R}^+$ such that $P(x)\in V$ whenever $x\in\mathcal{R}$
and $\Vert x\Vert<r$.
Take $x\in \mathcal{R}$ with $\Vert x\Vert<r^m$.
Since $\mathcal{R}$ is a $\ast$-subalgebra of $\mathcal{A}$, we see that
$\Re x,\Im x\in\mathcal{R}_{\textup{sa}}$. We write
$\Re x=x_1-x_2$ and $\Im x=x_3-x_4$, as in the proof of Theorem~\ref{1522}, where, on account of the condition (i),
\begin{align*}
x_1
&=\tfrac{1}{2}\bigl(\vert \Re x\vert+\Re x\bigr)\in\mathcal{R}_+,&
x_2&
=\tfrac{1}{2}\bigl(\vert\Re x\vert-\Re x\bigr) \in\mathcal{R}_+,\\
x_3
&=\tfrac{1}{2}\bigl(\vert \Im x\vert+\Im x\bigr)\in\mathcal{R}_+,&
x_4
&=\tfrac{1}{2}\bigl(\vert\Im x\vert-\Im x\bigr)
\in\mathcal{R}_+.
\end{align*}
For each $j\in\{1,2,3,4\}$,
condition (ii) gives $x_j^{1/m}\in\mathcal{R}$, and, further,
we have
$\Vert x_{j}^{1/m}\bigr\Vert=\Vert x_j\Vert^{1/m}\le\Vert x\Vert^{1/m}<r$.
Hence
\begin{align*}
\Phi(x) & =
\Phi\bigl(
\bigl(x_{1}^{1/m}\bigr)^m-\bigl(x_{2}^{1/m}\bigr)^m+i\bigl(x_{3}^{1/m}\bigr)^m
 -i\bigl(x_{4}^{1/m}\bigr)^m\bigr) \\
& =
\Phi\bigl(\bigl(x_{1}^{1/m}\bigr)^m\bigr)-
\Phi\bigl(\bigl(x_{2}^{1/m}\bigr)^m\bigr)+
i\Phi\bigl(\bigl(x_{3}^{1/m}\bigr)^m\bigr)-
i\Phi\bigl(\bigl(x_{4}^{1/m}\bigr)^m\bigr) \\
& =
P\bigl(x_{1}^{1/m}\bigr)-P\bigl(x_{2}^{1/m}\bigr)+iP\bigl(x_{3}^{1/m}\bigr)-iP\bigl(x_{4}^{1/m}\bigr)\in
V+V+V+V\subset U,
\end{align*}
which establishes the continuity of $\Phi$. 
\end{proof}

\begin{theorem}
Let $\mathcal{A}$ be a $C^*$-algebra,
let $X$ be a locally convex space, and
let $P\colon \mathcal{A}\to X$ be a continuous $m$-homogeneous polynomial.
Then the following conditions are equivalent:
\begin{enumerate}
\item[(i)]
there exists a continuous linear map $\Phi\colon\mathcal{A} \to X$ such that $P(x)=\Phi(x^m)$  $(x\in\mathcal{A})$;
\item[(ii)]
the polynomial $P$ is orthogonally additive on $\mathcal{A}_{\textup{sa}}$;
\item[(iii)]
the polynomial $P$ is orthogonally additive on $\mathcal{A}_+$.
\end{enumerate}
If the conditions are satisfied, then the map $\Phi$ is unique.
\end{theorem}

\begin{proof}
Theorem~\ref{1522} gives (i)$\Rightarrow$(ii), and obviously (ii)$\Rightarrow$(iii).
The task is now to prove that (iii)$\Rightarrow$(i).

Suppose that (iii) holds.
For each continuous linear functional $\eta\colon X\to\mathbb{C}$, set $P_\eta=\eta\circ P$.
Then $P_\eta$ is a complex-valued continuous $m$-homogeneous polynomial. 
We claim that $P_\eta$ is orthogonally additive on $\mathcal{A}_{\textup{sa}}$.
Take $x,y\in\mathcal{A}_{\textup{sa}}$ with $x\perp y$. Then we can write
$x=x_+-x_-$ and $y=y_+-y_-$ with $x_+,x_-,y_+,y_-\in\mathcal{A}_+$ mutually orthogonal.
Define $f\colon\mathbb{C}^2\to\mathbb{C}$ by
\[
f(\alpha,\beta)=
P_\eta(x_++\alpha x_-+y_++\beta y_-)-
P_\eta(x_++\alpha x_-)-P_\eta(y_++\beta y_-)\quad (\alpha,\beta\in\mathbb{C}^2).
\]
Then $f$ is a complex polynomial function in two complex variables.
If $\alpha,\beta\in\mathbb{R}^+$, then $x_++\alpha x_-,y_++\beta y_-\in\mathcal{A}_+$ are mutually orthogonal,
and so, by hypothesis, $P(x_++\alpha x_-+y_++\beta y_-)=P(x_++\alpha x_-)+P(y_++\beta y_-)$.
This shows that $f(\alpha,\beta)=0$. Since $f$ vanishes on 
$\mathbb{R}^{+}\times\mathbb{R}^{+}$, it follows that
$f$ vanishes on $\mathbb{C}^2$, which, in particular, implies
\[
P_\eta(x+y)-P_\eta(x)-P_\eta(y)=f(-1,-1)=0.
\]
Having proved that $P_\eta$ is orthogonally additive on 
$\mathcal{A}_{\textup{sa}}$ we can apply
\cite[Theorem~2.8]{P} to obtain a unique continuous linear functional $\Phi_\eta$ on $\mathcal{A}$
such that
\begin{equation}\label{2928}
\eta\bigl(P(x)\bigr)=\Phi_\eta(x^m)\quad (x\in\mathcal{A}).
\end{equation}

Each $x\in\mathcal{A}$ can be written in the form $x_1^m+\cdots+x_k^m$
for suitable $x_1,\ldots,x_k\in\mathcal{A}$, and we define 
\[
\Phi(x)=\sum_{j=1}^kP(x_j).
\]
Our next goal is to show that $\Phi$ is well-defined.
Suppose that $x_1,\ldots,x_k\in\mathcal{A}$ are such that $x_1^m+\cdots+x_k^m=0$.
For each continuous linear functional $\eta$ on $X$, \eqref{2928} gives
\[
\eta\left(\sum_{j=1}^kP(x_j)\right)=
\sum_{j=1}^k\eta\bigl(P(x_j)\bigr)=
\sum_{j=1}^k\Phi_\eta(x_j^m)=
\Phi_\eta\left(\sum_{j=1}^k x_j^m\right)=
0.
\]
Since $X$ is locally convex, we conclude that $\sum_{j=1}^kP(x_j)=0$.

It is a simple matter to check that $\Phi$ is linear and, by definition, $P(x)=\Phi(x^m)$
$(x\in\mathcal{A})$.
The continuity of $\Phi$ then follows from Lemma \ref{l12}.

The uniqueness of the map $\Phi$ follows from Theorem~\ref{1522}(ii).
\end{proof}

The assumption that the space $X$ be locally convex can be removed by requiring that
the $C^*$-algebra $\mathcal{A}$ be sufficiently rich in projections.
The real rank zero is the most important existence of projections property in the 
theory of $C^*$-algebras.
We refer the reader to \cite[Section V.3.2]{B} and \cite[Section V.7]{DAV} for the basic properties
and examples of $C^*$-algebras of real rank zero.
This class of $C^*$-algebras contains the von Neumann algebras and the
$C^*$-algebras $\mathcal{K}(H)$ of all compact operators on any Hilbert space $H$.
Let us remark that every $C^*$-algebra of real rank zero has an approximate unit of projections
(but not necessarily increasing).

\begin{theorem}\label{1137}
Let $\mathcal{A}$ be a $C^*$-algebra of real rank zero,
let $X$ be a topological linear space, and
let $P\colon \mathcal{A}\to X$ be a continuous $m$-homogeneous polynomial.
Suppose that $\mathcal{A}$  has an increasing approximate unit of projections.
Then the following conditions are equivalent:
\begin{enumerate}
\item[(i)]
there exists a continuous linear map $\Phi\colon\mathcal{A} \to X$ such that $P(x)=\Phi(x^m)$  $(x\in\mathcal{A})$;
\item[(ii)]
the polynomial $P$ is orthogonally additive on $\mathcal{A}_{\textup{sa}}$;
\item[(iii)]
the polynomial $P$ is orthogonally additive on $\mathcal{A}_+$.
\end{enumerate}
If the conditions are satisfied, then the map $\Phi$ is unique.
\end{theorem}

\begin{proof}
Theorem~\ref{1522} gives (i)$\Rightarrow$(ii), and 
it is clear that (ii)$\Rightarrow$(iii).
We will henceforth prove that (iii)$\Rightarrow$(i).

We first note that such a map $\Phi$ is necessarily unique, because of Theorem~\ref{1522}(ii).

Suppose that (iii) holds and that $\mathcal{A}$ is unital. 
Let $\varphi\colon\mathcal{A}^m\to X$ be the symmetric $m$-linear map associated with $P$
and define $\Phi\colon\mathcal{A}\to X$ by
\[
\Phi(x)=\varphi(x,1,\ldots,1)\quad (x\in\mathcal{A}).
\]
Let $Q\colon\mathcal{A}\to X$ be the $m$-homogeneous polynomial defined by
\[
Q(x)=\Phi(x^m)\quad(x\in\mathcal{A}).
\]
We will prove that $P=Q$. On account of Lemma \ref{l1}, it suffices to show that $P(x)=Q(x)$ 
for each $x\in\mathcal{A}_{\textup{sa}}$.

First, consider the case where $x\in\mathcal{A}_{\textup{sa}}$ has finite spectrum, 
say $\{\rho_1,\ldots,\rho_k\}\subset\mathbb{R}$. This implies that $x$
can be written in the form
\[
x=\sum_{j=1}^k\rho_j e_j,
\]
where $e_1,\ldots,e_k\in\mathcal{A}$ are mutually orthogonal projections 
(specifically, the projection $e_j$ is defined by using the continuous functional calculus for $x$ by
$e_j=\chi_{\{\rho_j\}}(x)$ for each $j\in\{1,\ldots,k\}$). We also set $e_0=1-(e_1+\cdots+e_k)$, so that
the projections $e_0,e_1,\ldots,e_k$ are mutually orthogonal, and $\rho_0=0$.
We claim that if $j_1,\ldots,j_m\in\{0,\ldots,k\}$ and $j_l\ne j_{l'}$ for some $l,l'\in\{1,\ldots,m\}$, then
\begin{equation}\label{1035}
\varphi(e_{j_1},\ldots,e_{j_m})=0.
\end{equation}
Let $\Lambda_1=\bigr\{n\in\{1,\ldots,m\} : j_n=j_l\bigr\}$ 
and $\Lambda_2=\bigr\{n\in\{1,\ldots,m\} : j_n\ne j_l\bigr\}$.
For each $\alpha_1,\ldots,\alpha_m\in\mathbb{R}^+$,
the elements $\sum_{n\in\Lambda_1}\alpha_{n}e_{j_n}$
and $\sum_{n\in\Lambda_2}\alpha_{n}e_{j_n}$ are positive and mutually orthogonal,
so that the orthogonal additivity of $P$ on $\mathcal{A}_+$ gives
\[
P\left(
\sum_{n=1}^m\alpha_ne_{j_n}
\right)=
P\left(\sum_{n\in\Lambda_1}\alpha_n e_{j_n}\right)+
P\left(\sum_{n\in\Lambda_2}\alpha_n e_{j_n}\right).
\]
This implies that, for each linear functional $\eta\colon X\to\mathbb{C}$, 
the function $f\colon\mathbb{C}^m\to\mathbb{C}$  defined by
\[
f(\alpha_1,\ldots,\alpha_m)=
\eta\left(
P\Bigl(\sum_{n=1}^m\alpha_n e_{j_n}\Bigr)-
P\Bigl(\sum_{n\in\Lambda_1}\alpha_n e_{j_n}\Bigr)-
P\Bigl(\sum_{n\in\Lambda_2}\alpha_n e_{j_n}\Bigr)
\right),
\]
for all $\alpha_1,\ldots,\alpha_m\in\mathbb{C}$,
is a complex polynomial function in $m$ complex variables vanishing in 
$\bigl(\mathbb{R}^+\bigr)^m$.
Therefore $f$ vanishes on $\mathbb{C}^m$.
Moreover, we observe that the coefficient of the monomial $\alpha_1\cdots\alpha_m$ is given by 
$n!\eta\bigl(\varphi(e_{j_1},\ldots,e_{j_m})\bigr)$,
because both $\Lambda_1$ and $\Lambda_2$ are different from $\{1,\ldots,m\}$.
We thus get 
\[n!\eta\bigl(\varphi(e_{j_1},\ldots,e_{j_m})\bigr)=0 .\]
Since this identity holds for each linear functional $\eta$, our claim follows.
Property \eqref{1035} now leads to
\begin{equation*}
\begin{split}
P(x) 
& = 
\varphi\left(\sum_{j=1}^k\rho_{j} e_j,\ldots,\sum_{j=1}^k\rho_{j}e_j\right) =
\sum_{j_1, \ldots, j_m =1}^k\rho_{j_{1}}\cdots\rho_{j_{m}}\varphi\left(e_{j_1},\ldots, e_{j_m}\right) \\
&= 
\sum_{j=1}^{k} \rho_{j}^m\varphi\left(e_j,\ldots,e_j\right)
\end{split}
\end{equation*}
and
\begin{equation*}
\begin{split}
Q(x) 
&= 
\varphi\left(\Bigl(\sum_{j=0}^k\rho_j e_j \Bigr)^m,\sum_{j=0}^k e_j,\ldots, \sum_{j=0}^k e_j\right)
= 
\varphi\left( \sum_{j=0}^k \rho_{j}^m e_j , \sum_{j=0}^k e_j,\ldots,\sum_{j=0}^k e_j \right) \\
&= 
\sum_{j_1,\ldots, j_m =0}^k \rho_{j_1}^m \varphi\left(e_{j_1},\ldots ,e_{j_m}\right)
= 
\sum_{j=1}^k \rho_{j}^m \varphi\left( e_j , \ldots ,e_j \right).
\end{split}
\end{equation*}
We thus get $P(x)=Q(x)$.

Now suppose that $x\in\mathcal{A}_{\textup{sa}}$ is an arbitrary element.
Since $\mathcal{A}$ has real rank zero, it follows that there exists a sequence $(x_n)$ in 
$\mathcal{A}_{\textup{sa}}$ such that each $x_n$ has finite spectrum and $\lim x_n=x$.
On account of the above case, we have $P(x_n)=Q(x_n)$ ($n\in\mathbb{N}$), and 
the continuity of both $P$ and $Q$ now yields $P(x)=\lim P(x_n)=\lim Q(x_n)=Q(x)$,
as required.

We are now in a position to prove the non-unital case. 
By hypothesis, there exists an increasing approximate unit of projections $(e_\lambda)_{\lambda\in\Lambda}$.
For each $\lambda\in\Lambda$, set $\mathcal{A}_\lambda=e_\lambda\mathcal{A}e_\lambda$.
Then $\mathcal{A}_\lambda$ is a unital $C^*$-algebra (with identity $e_\lambda$) and
has real rank zero (because $\mathcal{A}_\lambda$ is a hereditary $C^*$-subalgebra of $\mathcal{A}$).
From what has previously been proved, it follows that there exists a unique continuous linear map
$\Phi_\lambda\colon\mathcal{A}_\lambda\to X$ such that
\begin{equation}\label{e18}
P(x)=\Phi_\lambda(x^m)\quad (x\in\mathcal{A}_\lambda). 
\end{equation}
Define
\[
\mathcal{R}=\bigcup_{\lambda\in\Lambda}\mathcal{A}_\lambda
\]
and, for each $x\in\mathcal{R}$, set
\[
\Phi(x)=\Phi_\lambda(x),
\]
where $\lambda\in\Lambda$ is such that $x\in\mathcal{A}_\lambda$.
We will show that $\Phi$ is well-defined. Suppose $\lambda,\mu\in\Lambda$ are such that
$x\in\mathcal{A}_\lambda\cap\mathcal{A}_\mu$. Then there exists $\nu\in\Lambda$ with
$\lambda,\mu\le\nu$. Since the net $(e_\lambda)_{\lambda\in\Lambda}$ is increasing, we see that $e_\lambda,e_\mu\le e_\nu$
and therefore $\mathcal{A}_\lambda,\mathcal{A}_\mu\subset\mathcal{A}_\nu$. The uniqueness of the
representation of $P$ on both $\mathcal{A}_\lambda$ and $\mathcal{A}_\mu$ implies that
$\Phi_\nu\mid_{\mathcal{A}_\lambda}=\Phi_\lambda$ and $\Phi_\nu\mid_{\mathcal{A}_\mu}=\Phi_\mu$,
which implies that $\Phi_\lambda(x)=\Phi_\nu(x)=\Phi_\mu(x)$.
We now show that $\mathcal{R}$ is a $\ast$-subalgebra of $\mathcal{A}$ and that $\Phi$ is linear.
Take $x,y\in\mathcal{R}$ and $\alpha,\beta\in\mathbb{C}$. We take $\lambda,\mu\in\Lambda$
such that $x\in\mathcal{A}_{\lambda}$ and $y\in\mathcal{A}_\mu$. 
Then $x^*\in\mathcal{A}_\lambda\subset\mathcal{R}$. Now set $\nu\in\Lambda$ with
$\lambda,\mu\le\nu$. Hence $x,y\in\mathcal{A}_\nu$, so that
$\alpha x+\beta y,xy\in\mathcal{A}_\nu\subset\mathcal{R}$, which shows that $\mathcal{R}$ is a
subalgebra of $\mathcal{A}$. Further, we have
\[
\Phi(\alpha x+\beta y)=\Phi_\nu(\alpha x+\beta y)=
\alpha\Phi_\nu(x)+\beta\Phi_\nu(y)=\alpha \Phi(x)+\beta\Phi(y),
\]
which shows that $\Phi$ is linear.

From \eqref{e18} we deduce  that $P(x)=\Phi(x^m)$ for each $x\in\mathcal{R}$.

Our next goal is to show that $\mathcal{R}$ satisfies the conditions of Lemma \ref{l12}.
If $x\in\mathcal{R}_{\textup{sa}}$ ($x\in\mathcal{R}_+$), 
then there exists $\lambda\in\Lambda$ such that $x\in\bigl(\mathcal{A}_\lambda\bigr)_{\textup{sa}}$
($x\in\bigl(\mathcal{A}_\lambda\bigr)_{+}$, respectively)
and therefore $\vert x\vert\in\mathcal{A}_\lambda\subset\mathcal{R}$ ($x^{1/m}\in \mathcal{A}_\lambda\subset\mathcal{R}$,
respectively). Since the polynomial $P\mid_\mathcal{R}$ is continuous, Lemma \ref{l12}
shows that the map $\Phi$ is continuous.

Since $(e_\lambda)_{\lambda\in\Lambda}$ is an approximate unit, it follows that $\mathcal{R}$ is dense in $\mathcal{A}$,
and hence that the map $\Phi$ extends uniquely to a continuous linear map from $\mathcal{A}$ into
the completion of $X$. By abuse of notation we continue to write $\Phi$ for this extension.
Since both $P$ and $\Phi$ are continuous, it may be concluded that
$P(x)=\Phi(x^m)$ for each $x\in\mathcal{A}$.
We next prove that the image of $\Phi$ is actually contained in $X$.
Of course, it suffices to show that $\Phi$ takes $\mathcal{A}_+$ into $X$.
If $x\in\mathcal{A}_+$, then
\[
\Phi(x)=
\Phi\bigl(\bigl(x^{1/m}\bigr)^m\bigr)=
P\bigl(x^{1/m}\bigr)\in X,
\]
as required.
\end{proof}

Since every von Neumann algebra is unital and has real rank zero,
Theorem~\ref{1137} applies in this setting and gives the following.

\begin{corollary}\label{1523}
Let $\mathcal{M}$ be a von Neumann algebra,
let $X$ be a topological linear space, and
let $P\colon \mathcal{M}\to X$ be a continuous $m$-homogeneous polynomial.
Then the following conditions are equivalent:
\begin{enumerate}
\item[(i)]
there exists a continuous linear map $\Phi\colon\mathcal{M} \to X$ such that $P(x)=\Phi(x^m)$  $(x\in\mathcal{M})$;
\item[(ii)]
the polynomial $P$ is orthogonally additive on $\mathcal{M}_{\textup{sa}}$;
\item[(iii)]
the polynomial $P$ is orthogonally additive on $\mathcal{M}_+$.
\end{enumerate}
If the conditions are satisfied, then the map $\Phi$ is unique.
\end{corollary}

\begin{proposition}\label{1525}
Let $H$ be a Hilbert space with $\dim H \geq 2$,
let $X$ be a topological linear space, and 
let $P\colon\mathcal{B}(H) \to X$ be a continuous $m$-homogeneous polynomial.
Suppose that $P$ is orthogonally additive in $\mathcal{B}(H)$.
Then $P=0$.
\end{proposition}

\begin{proof}
For each unitary $v\in\mathcal{B}(H)$, the map $P_v\colon\mathcal{B}(H)\to X$ defined by
\[
P_v(x)=P(vx)\quad (x\in\mathcal{B}(H))
\]
is easily seen to be a continuous $m$-homogeneous polynomial that is orthogonally additive on $\mathcal{B}(H)$.
In particular, $P_v$ is orthogonally additive on $\mathcal{B}(H)_{\textup{sa}}$, and Corollary~\ref{1523} then
gives a unique continuous linear map $\Phi_v\colon\mathcal{B}(H)\to X$ such that
\[
P(vx)=\Phi_v(x^m)\quad (x\in\mathcal{B}(H)).
\]

We claim that, 
if $e,e'\in\mathcal{B}(H)$ are equivalent projections with $e\perp e'$, 
then $P(e)=P(e')=0$.
Let $u\in\mathcal{B}(H)$ be a partial isometry such that $u^*u=e$ and $uu^*=e'$.
Then
\[
\left\Vert u^2 \right\Vert^4=
\left\Vert (u^2)^*u^2 \right\Vert^2=
\left\Vert \left((u^2)^*u^2\right)^2\right\Vert=
\left\Vert u^{*}ee^{\prime} eu \right\Vert=0,
\]
which gives $u^2=0$.
From this we see that $u\perp u^*$, and therefore
\begin{equation}\label{1630}
P(vu+vu^*)=
P_v(u+u^*)=P_v(u)+P_v(u^*)=\Phi_v\left(u^m\right)+\Phi_v\left((u^*)^m\right)=0.
\end{equation}
We now take $\omega\in\mathbb{C}$ with $\omega^m=-1$, and define
\begin{equation*}
\begin{split}
v & =
1+u+u^*-e-e',\\
v_\omega
& = 
1+\omega u+u^*-e-e'.
\end{split}
\end{equation*}
It is immediately seen that both $v$ and $v_\omega$ are unitary, and so applying \eqref{1630} 
(and using the orthogonal additivity of $P$ and that $e\perp e'$), we see that
\begin{align*}
0 & =P(vu+vu^*)=P(e+e')=P(e)+P(e'),\\
0 & =P(v_\omega u+v_\omega u^*)=P(e+\omega e')=P(e)+P(\omega e')=P(e)-P(e').
\end{align*}
By comparing both identities, we conclude that $P(e)=P(e')=0$, as claimed.

Our next objective is to prove that $P(e)=0$ for each projection $e\in\mathcal{B}(H)$.
Suppose that $e\in\mathcal{B}(H)$ is a rank-one projection. 
Since $\dim H\ge 2$, it follows that there exists an equivalent projection $e'$ such that
$e'\perp e$. 
Then it follows from the above claim that $P(e)=0$.
Let $e\in\mathcal{B}(H)$ be a finite projection. Then there exist mutually orthogonal projections
$e_1,\ldots,e_n$ such that $e_1+\cdots+e_n=e$. Using the preceding observation and the orthogonal
additivity of $P$ we get
$P(e)=P(e_1)+\cdots+P(e_n)=0$.
We now assume that $e\in\mathcal{B}(H)$ is an infinite projection.
Then there exist mutually orthogonal, equivalent projections $e_1$ and $e_2$ such that $e_1+e_2=e$.
By the claim, 
we have $P(e)=P(e_1)+P(e_2)=0$.

We finally proceed to show that $P=0$.
By Lemma \ref{l1}, it suffices to show that $P(x)=0$ for each $x\in\mathcal{B}(H)_+$.
Suppose that $x\in\mathcal{B}(H)_+$  can be written in the form
$x=\sum_{j=1}^k\rho_j e_j$,
where $e_1,\ldots,e_k\in\mathcal{B}(H)$ are mutually orthogonal projections and
$\rho_1,\ldots,\rho_k\in\mathbb{R}^+$. Then we have
$P(x)=\sum_{j=1}^k{\rho_j}^mP(e_j)=0$.
Now let $x\in\mathcal{B}(H)_{+}$ be an arbitrary element.
From the spectral decomposition we deduce  that there exists a sequence $(x_n)$ in $\mathcal{B}(H)_{+}$
such that each $x_n$ is a positive linear combination of mutually orthogonal projections and $\lim x_n=x$.
On account of the preceding observation, $P(x_n)=0$ ($n\in\mathbb{N}$), and 
the continuity of $P$ implies that
$P(x)=\lim P(x_n)=0$, as required.
\end{proof}

\section{Non-commutative $L^p$-spaces}

Before giving the next results we make the following preliminary remarks.

A fundamental fact for us is the behaviour of the product of $L^0(\mathcal{M},\tau)$ when restricted to the $L^p$-spaces.
Specifically, if $0<p,q,r\le\infty$ are such that $\tfrac{1}{p}+\tfrac{1}{q}=\tfrac{1}{r}$,
then the H\"older inequality states that
\begin{equation}\label{1723}
x\in L^p(\mathcal{M},\tau), \
y\in L^q(\mathcal{M},\tau)  \
\Rightarrow \
xy\in L^r(\mathcal{M},\tau) \text{ and }
\Vert xy\Vert_r\le\Vert x\Vert_p\Vert y\Vert_q.
\end{equation}

Suppose that $x,y\in L^p(\mathcal{M},\tau)_+$, $0<p<\infty$, are mutually orthogonal and that 
$\omega\in\mathbb{C}$ with $\vert\omega\vert=1$. Then it is immediately seen that
$\vert x+\omega y\vert=x+y$, and it follows, by considering the spectral resolutions of $x$, $y$,
and $x+y$, that $(x+y)^p=x^p+y^p$. Hence
\begin{equation}\label{1150}
\Vert x+\omega y\Vert_p^p=\Vert x\Vert_p^p+\Vert y\Vert_p^p.
\end{equation}
Each $x\in L^{p}(\mathcal{M},\tau)$ can be written in the form
\begin{equation}\label{1813}
\begin{split}
x=x_1-x_2+i(x_3-x_4), \text{ with } 
&x_1,x_2,x_3,x_4\in L^{p}(\mathcal{M},\tau)_+, \\
&x_1\perp x_2, \ x_3\perp x_4,\\
&\Vert x_1\Vert_p^p+\Vert x_2\Vert_p^p=\Vert x_1-x_2\Vert_p^p\le\Vert x\Vert_p^p,\\
&\Vert x_3\Vert_p^p+\Vert x_4\Vert_p^p=\Vert x_3-x_4\Vert_p^p\le\Vert x\Vert_p^p.
\end{split}
\end{equation}
Indeed, first we write $x=\Re x+i\Im x$, where
\[
\Re x=\frac{1}{2}(x^*+x), \, \Im x=\frac{i}{2}(x^*-x)\in L^{p}(\mathcal{M},\tau)_{\textup{sa}},
\]
and, since $\Vert x^*\Vert_{p}=\Vert x\Vert_{p}$, it follows that
$\Vert \Re x\Vert_{p},\Vert\Im x\Vert_{p}\le\Vert x\Vert_{p}$.
Further, we take the positive operators
\begin{equation*}
x_1
=\tfrac{1}{2}\left(\vert \Re x\vert+\Re x \right),\
x_2
=\tfrac{1}{2} \left(\vert \Re x\vert-\Re x \right),\
x_3
=\tfrac{1}{2} \left(\vert \Im x\vert+\Im x \right),\
x_4
=\tfrac{1}{2} \left(\vert \Im x\vert-\Im x \right).
\end{equation*}
Then $x_1,x_2,x_3,x_4\in L^p(\mathcal{M},\tau)$,
$\Re x=x_1-x_2$ with $x_1\perp x_2$, so that \eqref{1150} gives
\[
\Vert\Re x\Vert_p^p=\Vert x_1\Vert_p^p+\Vert x_2\Vert_p^p,
\]
and $\Im x= x_3-x_4$ with $x_3\perp x_4$, so that \eqref{1150} gives
\[
\Vert\Im x\Vert_p^p=\Vert x_3\Vert_p^p+\Vert x_4\Vert_p^p.
\]

\begin{theorem}\label{1242}
Let $\mathcal{M}$ be a von Neumann algebra with a normal semifinite faithful trace $\tau$,
let $X$ be a topological linear space, and
let $\Phi\colon L^{p/m}(\mathcal{M},\tau)\to X$ be a continuous linear map with $0<p<\infty$.
Then: 
\begin{enumerate}
\item[(i)]
the map $P_\Phi\colon L^p(\mathcal{M},\tau)\to X$ defined by
$P_\Phi(x)=\Phi(x^m)$ $(x\in L^p(\mathcal{M},\tau))$ is a continuous $m$-homogeneous polynomial which
is orthogonally additive on $L^p(\mathcal{M},\tau)_{\textup{sa}}$;
\item[(ii)]
the polynomial $P_\Phi$ is uniquely specified by the map $\Phi$.
\end{enumerate}
Suppose, further, that $X$ is a $q$-normed space, $0<q\le 1$. 
Then:
\begin{enumerate}
\item[(iii)]
$2^{-1/q}\Vert\Phi\Vert\le\Vert P_\Phi\Vert\le\Vert\Phi\Vert$.
\end{enumerate}
Moreover, in the case where $X=\mathbb{C}$,
\begin{enumerate}
\item[(iv)]
the functional $\Phi$ is hermitian if and only if the polynomial $P_\Phi$
is hermitian, in which case $\Vert P_\Phi\Vert=\Vert\Phi\Vert$.
\end{enumerate}
\end{theorem}

\begin{proof}
The proof of this result is similar to that establishing Theorem~\ref{1522}.

(i)
It follows immediately from \eqref{1723} that, for each $x_1,\ldots,x_m\in L^p(\mathcal{M},\tau)$,
\begin{equation}\label{1724}
x_1\cdots x_m\in L^{p/m}(\mathcal{M},\tau)  \text{ and }
\Vert x_1\cdots x_m\Vert_{p/m}\le\Vert x_1\Vert_p\cdots\Vert x_m\Vert_p.
\end{equation}
On the one hand,
this clearly implies that the map $P_\Phi$ is well-defined,
on the other hand,
the map  $x\mapsto x^m$ from $L^p(\mathcal{M},\tau)$ into $L^{p/m}(\mathcal{M},\tau)$ is continuous, and
so $P_\Phi$ is continuous. Further, $P_\Phi$ is the $m$-homogeneous polynomial associated with
the symmetric $m$-linear map $\varphi\colon L^p(\mathcal{M},\tau)^m\to X$ defined by
\[
\varphi(x_1,\ldots,x_m)= 
\frac{1}{m!}\sum_{\sigma \in \mathfrak{S}_m} 
\Phi\left(x_{\sigma(1)} \cdots x_{\sigma(m)} \right) \quad (x_1,\ldots,x_m\in L^p(\mathcal{M},\tau)).
\]
Suppose that $x,y\in L^p(\mathcal{M},\tau)_{\textup{sa}}$ are such that $x\perp y$. 
Then $xy=yx=0$, and so
$(x+y)^m=x^m+y^m$, which gives
\[
P_\Phi(x+y)=\Phi\bigl((x+y)^m\bigr)=
\Phi\bigl(x^m+y^m\bigr)=\Phi\bigl(x^m\bigr)+\Phi\bigl(y^m\bigr)=
P_\Phi(x)+P_\Phi(y).
\]

(ii)
Suppose that $\Psi\colon L^{p/m}(\mathcal{M},\tau)\to X$ is a linear map such that $P_\Psi=P_\Phi$.
For each $x\in L^{p/m}(\mathcal{M},\tau)_+$, we have $x^{1/m}\in L^{p}(\mathcal{M},\tau)$ and
\[
\Phi(x)=
\Phi\bigl(\bigl(x^{1/m}\bigr)^m\bigr)=
P\bigl(x^{1/m}\bigr)=
\Psi\bigl(\bigl(x^{1/m}\bigr)^m\bigr)=
\Psi(x).
\]
By linearity we obtain $\Phi=\Psi$.

(iii)
Next, assume that $X$ is a $q$-normed space.
For each $x\in L^p(\mathcal{M},\tau)$, by \eqref{1724}, we have
\[
\Vert P_\Phi(x)\Vert=\Vert\Phi(x^m)\Vert\le\Vert\Phi\Vert\Vert x^m\Vert_{p/m}\le\Vert\Phi\Vert\Vert x\Vert_p^m,
\] 
which clearly implies that $\Vert P_\Phi\Vert\le\Vert\Phi\Vert$.
Now take $x\in L^{p/m}(\mathcal{M},\tau)$, and take $\omega\in\mathbb{C}$ with $\omega^m=-1$.
Write 
\[
x=
\Re x+i\Im x=
x_1-x_2+i(x_3-x_4)
\] 
as in \eqref{1813} (with $p/m$ instead of $p$).
Since $x_1\perp x_2$ and $x_3\perp x_4$, it follows that
$x_1^{1/m}\perp x_2^{1/m}$ and $x_3^{1/m}\perp x_4^{1/m}$, so that \eqref{1150} gives
\begin{equation}\label{2025}
\begin{split}
\Vert \Re x\Vert_{p/m}^{p/m}
&=
\Vert x_1\Vert_{p/m}^{p/m}+\Vert x_2\Vert_{p/m}^{p/m},\\
\Vert \Im x\Vert_{p/m}^{p/m}
&=
\Vert x_3\Vert_{p/m}^{p/m}+\Vert x_4\Vert_{p/m}^{p/m},
\end{split}
\end{equation}
and
\begin{equation}\label{2028}
\begin{split}
\bigl\Vert
x_1^{1/m}+\omega x_2^{1/m}
\bigr\Vert_p^p & =
\bigl\Vert x_1^{1/m}\bigr\Vert_p^p+\bigr\Vert x_2^{1/m}
\bigr\Vert_p^p,\\
\bigl\Vert
x_3^{1/m}+\omega x_4^{1/m}
\bigr\Vert_p^p & =
\bigl\Vert x_3^{1/m}\bigr\Vert_p^p+\bigr\Vert x_4^{1/m}
\bigr\Vert_p^p.
\end{split}
\end{equation}
Further, we have $x_1^{1/m},x_2^{1/m},x_3^{1/m},x_4^{1/m}\in L^p(\mathcal{M},\tau)$ and
\[
\bigl\Vert x_1^{1/m}\bigr\Vert_p=
\Vert x_1\Vert_{p/m}^{1/m}, \
\bigl\Vert x_2^{1/m}\bigr\Vert_p=
\Vert x_2\Vert_{p/m}^{1/m}, \
\bigl\Vert x_3^{1/m}\bigr\Vert_p=
\Vert x_3\Vert_{p/m}^{1/m}, \
\bigl\Vert x_4^{1/m}\bigr\Vert_p=
\Vert x_4\Vert_{p/m}^{1/m},
\]
so that \eqref{2025} and \eqref{2028} give
\begin{equation}\label{2029}
\begin{split}
\bigl\Vert x_1^{1/m}+\omega x_2^{1/m}\bigr\Vert_p^p
&=
\Vert \Re x\Vert_{p/m}^{p/m},\\
\bigl\Vert x_3^{1/m}+\omega x_4^{1/m}\bigr\Vert_p^p
&=
\Vert \Im x\Vert_{p/m}^{p/m}.
\end{split}
\end{equation}
On the other hand, we have
\begin{equation*}
\bigl(x_1^{1/m}+\omega x_2^{1/m}\bigr)^m=x_1-x_2=\Re x,\quad
\bigl(x_3^{1/m}+\omega x_4^{1/m}\bigr)^m=x_3-x_4=\Im x,
\end{equation*}
whence
\begin{align*}
\Phi(x) & =
\Phi(\Re x)+i\Phi(\Im x)=
\Phi\bigl(\bigl(x_1^{1/m}+\omega x_2^{1/m}\bigr)^m\bigr)+
i\Phi\bigl(\bigl(x_3^{1/m}+\omega x_4^{1/m}\bigr)^m\bigr) \\
& =
P_\Phi\bigl(x_1^{1/m}+\omega x_2^{1/m}\bigr)+
iP_\Phi\bigl(x_3^{1/m}+\omega x_4^{1/m}\bigr).
\end{align*}
Hence, by \eqref{2029},
\begin{equation*}
\begin{split}
\Vert\Phi(x)\Vert^q
&\le
\bigl\Vert P_\Phi\bigl(x_1^{1/m}+\omega x_2^{1/m}\bigr)\bigr\Vert^q+
\bigl\Vert P_\Phi\bigl(x_3^{1/m}+\omega x_4^{1/m}\bigr)\bigr\Vert^q\\
&\le
\Vert P_\Phi\Vert^q\bigl\Vert x_1^{1/m}+\omega x_2^{1/m}\bigr\Vert_p^{mq}+
\Vert P_\Phi\Vert^q\bigl\Vert x_3^{1/m}+\omega x_4^{1/m}\bigr\Vert_p^{mq}\\
&=
\Vert P_\Phi\Vert^q\left(
\Vert \Re x\Vert^q+\Vert \Im x\Vert^q\right)\\
&\le
\Vert P_\Phi\Vert^q 2\Vert x\Vert^q.
\end{split}
\end{equation*}
This clearly forces $\Vert\Phi\Vert\le 2^{1/q}\Vert P_\Phi\Vert$, as claimed.

(iv)
It is straightforward to check that $P_\Phi^*=P_{\Phi^*}$. 
From this deduce that $\Phi$ is hermitian if and only if $P_\Phi$ is hermitian as in the proof of Theorem~\ref{1522}(iv).
Suppose that $\Phi$ is a hermitian functional.
By direct calculation, we see that $P_\Phi$ is hermitian, and it remains to prove that
$\Vert P_\Phi\Vert=\Vert\Phi\Vert$. We only need to show that
$\Vert\Phi\Vert\le\Vert P_\Phi\Vert$. To this end,
let $\varepsilon\in\mathbb{R}^+$, 
and choose $x\in L^{p/m}(\mathcal{M},\tau)$ such that
$\Vert x\Vert_{p/m}=1$ and 
$\Vert\Phi\Vert-\varepsilon<\vert\Phi(x)\vert$.
We take $\alpha\in\mathbb{C}$ with $\vert\alpha\vert=1$ and 
$\vert\Phi(x)\vert=\alpha\Phi(x)$, so that
\[
\Vert\Phi\Vert-\varepsilon<\vert\Phi(x)\vert=\Phi(\alpha x)=
\overline{\Phi(\alpha x)}=\Phi\bigl((\alpha x)^*\bigr).
\]
We see that
$\Re(\alpha x)\in L^{p/m}(\mathcal{M},\tau)_{\textup{sa}}$,
$\Vert \Re(\alpha x)\Vert_{p/m}\le 1$, and 
$\Vert\Phi\Vert-\varepsilon<\Phi(\Re(\alpha x))$.
Now we consider the decomposition $\Re(\alpha x)=x_1-x_2$ as in \eqref{1813} (with $p/m$ instead of $p$), 
and take $\omega\in\mathbb{C}$ with $\omega^m=-1$.
As in \eqref{2029}, we see that
$\bigl\Vert x_1^{1/m}+\omega x_2^{1/m}\bigr\Vert=\Vert \Re(\alpha x)\Vert^{1/m}\le 1$.
Moreover, we have
\[
P_\Phi\bigl( x_1^{1/m}+\omega x_2^{1/m}\bigr)=
\Phi\bigl(\bigl( x_1^{1/m}+\omega x_2^{1/m}\bigr)^m\bigr)=
\Phi(\Re(\alpha x)),
\]
and so $\Vert\Phi\Vert-\varepsilon<\Vert P_\Phi\Vert$.
\end{proof}

\begin{theorem}\label{1527}
Let $\mathcal{M}$ be a von Neumann algebra with a normal semifinite faithful trace $\tau$,
let $X$ be a topological linear space, and 
let $P\colon L^p(\mathcal{M},\tau)\to X$ be a continuous $m$-homogeneous polynomial
with $0<p<\infty$.
Then the following conditions are equivalent:
\begin{enumerate}
\item[(i)]
there exists a continuous linear map 
$\Phi\colon L^{p/m}(\mathcal{M},\tau)\to X$ such that $P(x)=\Phi(x^m)$  $(x\in L^p(\mathcal{M},\tau))$;
\item[(ii)]
the polynomial $P$ is orthogonally additive on $L^p(\mathcal{M},\tau)_{\textup{sa}}$;
\item[(iii)]
the polynomial $P$ is orthogonally additive on $S(\mathcal{M},\tau)_+$.
\end{enumerate}
If the conditions are satisfied, then the map $\Phi$ is unique.
\end{theorem}

\begin{proof}
Theorem~\ref{1242} shows that (i)$\Rightarrow$(ii), and
it is obvious that (ii)$\Rightarrow$(iii).
We proceed to prove that (iii)$\Rightarrow$(i).

Suppose that (iii) holds.
Let $e\in\mathcal{M}$ be a projection such that $\tau(e)<\infty$,
and consider the von Neumann algebra $\mathcal{M}_e=e\mathcal{M}e$.
We claim that $\mathcal{M}_e\subset S(\mathcal{M},\tau)$ and that
there exists a unique continuous linear map $\Phi_e\colon\mathcal{M}_e\to X$ such that
\begin{equation}\label{1904}
P(x)=\Phi_e(x^m)\quad(x\in\mathcal{M}_e).
\end{equation} 
Set $x\in\mathcal{M}_e$, and write
$x=(x_1-x_2)+i(x_3-x_4)$ with $x_1,x_2,x_3,x_4\in\mathcal{M}_{e \, +}$.
Then $\soporte (x_j)\le e$ and therefore $\tau\bigl(\soporte (x_j)\bigr)\le\tau(e)<\infty$ $(j\in\{1,2,3,4\})$.
This shows that $x_j\in S(\mathcal{M},\tau)$ $(j\in\{1,2,3,4\})$, whence $x\in S(\mathcal{M},\tau)$.
Our next goal is to show that the restriction $P\mid_{\mathcal{M}_e}$ is continuous 
(with respect to the norm that $\mathcal{M}_e$ inherits as a closed subspace of $\mathcal{M}$).
Let $x\in\mathcal{M}_e$, and let $U\subset X$ be a neighbourhood of $P(x)$.
Since $P$ is continuous, the set $P^{-1}(U)$ is a neighbourhood of $x$ in $L^p(\mathcal{M},\tau)$,
which implies that
there exists $r\in\mathbb{R}^+$ such that $P(y)\in U$ whenever $y\in L^p(\mathcal{M},\tau)$
and $\Vert y- x\Vert_p<r$. If $y\in\mathcal{M}_e$ is such that $\Vert y- x\Vert<r/\Vert e\Vert_p$, then,
from \eqref{1723}, we obtain
\[
\Vert y-x\Vert_p=\Vert e(y-x)\Vert_p\le\Vert e\Vert_p\Vert y- x\Vert<r
\] 
and therefore $P(y)\in U$.
Hence $P\mid_{\mathcal{M}_e}$ is continuous. Since, by hypothesis, the polynomial $P\mid_{\mathcal{M}_e}$
is orthogonally additive on $\mathcal{M}_{e \, +}$, Corollary~\ref{1523} states that there exists a unique
continuous linear map $\Phi_e\colon\mathcal{M}_e\to X$ such that \eqref{1904} holds.

For each $x\in S(\mathcal{M},\tau)$, define
\[
\Phi(x)=\Phi_e(x),
\]
where $e\in\mathcal{M}$ is any projection such that
\begin{equation}\label{1905}
ex=xe=x\quad\text{and}\quad \tau(e)<\infty.
\end{equation}
We will show that $\Phi$ is well-defined.
For this purpose we first check that, if $x\in S(\mathcal{M},\tau)$, 
then there exists a projection $e$ such that \eqref{1905} holds.
Indeed, 
we write $x=\sum_{j=1}^k\alpha_j x_j$ with
$\alpha_1,\ldots,\alpha_k\in\mathbb{C}$ and $x_1,\ldots,x_k\in S(\mathcal{M},\tau)_+$,
and define $e=\soporte (x_1)\vee\cdots\vee\soporte (x_k)$. Then $ex=xe=x$ and
$\tau(e)\le\sum_{j=1}^k\tau\bigl(\soporte (x_j)\bigr)<\infty$, as required.
Suppose that $x\in S(\mathcal{M},\tau)$ and that $e_1,e_2\in\mathcal{M}$ are projections satisfying \eqref{1905}.
Then the projection $e=e_1\vee e_2$ satisfies \eqref{1905} and $\mathcal{M}_{e_1},\mathcal{M}_{e_2}\subset\mathcal{M}_e$.
The uniqueness of the representation \eqref{1904} on both $\mathcal{M}_{e_1}$ and $\mathcal{M}_{e_2}$ gives
$\Phi_{e}\mid_{\mathcal{M}_{e_1}}=\Phi_{e_1}$ and
$\Phi_{e}\mid_{\mathcal{M}_{e_2}}=\Phi_{e_2}$, which implies that
$\Phi_{e_1}(x)=\Phi_{e}(x)=\Phi_{e_2}(x)$.

We now show that $\Phi$ is linear.
Take $x_1,x_2\in S(\mathcal{M},\tau)$ and $\alpha,\beta\in\mathbb{C}$.
Let $e_1,e_2\in\mathcal{M}$ be projections such that $e_jx_j=x_je_j=x_j$
and $\tau(e_j)<\infty$ $(j\in\{1,2\})$.
Then the projection $e=e_1\vee e_2$ satisfies
\begin{gather*}
ex_j=x_je=x_j\quad (j\in\{1,2\}), \\
e(\alpha x_1+\beta x_2)=(\alpha x_1+\beta x_2)e=\alpha x_1+\beta x_2,
\end{gather*}
and
\[
\tau(e)\le\tau(e_1)+\tau(e_2)<\infty.
\] 
Thus
\[
\Phi(x_j)=\Phi_e(x_j)\quad (j\in\{1,2\})
\]
and
\[
\Phi(\alpha x_1+\beta x_2)=
\Phi_e(\alpha x_1+\beta x_2)=
\alpha\Phi_e(x_1)+\beta\Phi_e(x_2)=
\alpha\Phi(x_1)+\beta\Phi(x_2).
\]

We see from the definition of $\Phi$ that
\begin{equation}\label{1906}
P(x)=\Phi(x^m) \quad (x\in S(\mathcal{M},\tau)).
\end{equation}

Our next concern will be the continuity of $\Phi$ with respect to the norm $\Vert\cdot\Vert_{p/m}$.
Let $U$ be a neighbourhood of $0$ in $X$. Let $V$ be a balanced neighbourhood of $0$ in $X$ with
$V+V+V+V\subset U$. 
The set $P^{-1}(V)$ is a neighbourhood of $0$ in $L^p(\mathcal{M},\tau)$,
which implies that there exists $r\in\mathbb{R}^+$ such that $P(x)\in V$ whenever $x\in L^p(\mathcal{M},\tau)$
and $\Vert x\Vert_p<r$.
Take $x\in S(\mathcal{M},\tau)$ with $\Vert x\Vert_{p/m}<r^m$, and write 
$x=(x_1-x_2)+i(x_3-x_4)$ as in \eqref{1813} (with $p/m$ instead of $p$). 
Then it is immediate to check that actually
$x_1,x_2,x_3,x_4\in S(\mathcal{M},\tau)_+$ and, further,
$\Vert x_j\Vert_{p/m}\le\Vert x\Vert_{p/m}$ $(j\in\{1,2,3,4\})$.
For each $j\in\{1,2,3,4\}$, we have
\begin{align*}
\bigl\Vert x_{j}^{1/m}\bigr\Vert_p & =
\tau\bigl(x_{j}^{p/m}\bigr)^{1/p}= 
\bigl(\tau\bigl(x_{j}^{p/m}\bigr)^{m/p}\bigr)^{1/m}=
{\Vert x_j\Vert_{p/m}}^{1/m} \\
& \le
\Vert x\Vert_{p/m}^{1/m}<r,
\end{align*}
whence
\begin{align*}
\Phi(x) & =
\Phi\left(
\bigl(x_{1}^{1/m}\bigr)^m-\bigl(x_{2}^{1/m}\bigr)^m+i\bigl(x_{3}^{1/m}\bigr)^m-i\bigl(x_{4}^{1/m}\bigr)^m\right) \\
& =
\Phi\left(\bigl(x_{1}^{1/m}\bigr)^m\right)-
\Phi\left(\bigl(x_{2}^{1/m}\bigr)^m\right)+
i\Phi\left(\bigl(x_{3}^{1/m}\bigr)^m\right)-
i\Phi\left(\bigl(x_{4}^{1/m}\bigr)^m\right) \\
& =
P\bigl(x_{1}^{1/m}\bigr)-P\bigl(x_{2}^{1/m}\bigr) \\
& \quad {}+iP\bigl(x_{3}^{1/m}\bigr)-iP\bigl(x_{4}^{1/m}\bigr)\in
V+V+V+V\subset U,
\end{align*}
which establishes the continuity of $\Phi$. 
Since $S(\mathcal{M},\tau)$ is dense in $L^{p/m}(\mathcal{M},\tau)$,
the map $\Phi$ extends uniquely to a continuous linear map from $L^{p/m}(\mathcal{M},\tau)$ into
the completion of $X$. By abuse of notation we continue to write $\Phi$ for this extension.
Since both $P$ and $\Phi$ are continuous, \eqref{1906} gives
$P(x)=\Phi(x^m)$ for each $x\in L^p(\mathcal{M})$.
The task is now to show that the image of $\Phi$ is actually contained in $X$.
Of course, it suffices to show that $\Phi$ takes $L^{p/m}(\mathcal{M},\tau)_+$ into $X$.
Let $x\in L^{p/m}(\mathcal{M},\tau)_+$.
Then $x^{1/m}\in L^{p}(\mathcal{M},\tau)_+$
and
\[
\Phi(x)=
\Phi\bigl(\bigl(x^{1/m}\bigr)^m\bigr)=
P\bigl(x^{1/m}\bigr)\in
X,
\]
as required.

The uniqueness of the map $\Phi$ is given by Theorem~\ref{1242}(ii).
\end{proof}

Let us note that the space of all continuous $m$-homogeneous polynomials from $L^p(\mathcal{M},\tau)$ 
into any topological linear space $X$ which are orthogonally additive on $S(\mathcal{M},\tau)_+$ is 
sufficiently rich in the case where $p/m\ge 1$, because of the existence of continuous linear functionals 
on $L^{p/m}(\mathcal{M},\tau)$.
However, some restriction on the space $X$ must be imposed when we consider the case $p/m<1$
and the von Neumann algebra $\mathcal{M}$ has no minimal projections, because in this case the
dual of $L^{p/m}(\mathcal{M},\tau)$ is trivial (\cite{Sa}). In fact, there are no non-zero continuous
linear maps from $L^p(\mathcal{M},\tau)$ into any $q$-normed space $X$ with $q>p$.
We think that this property is probably well-known, but we have not been able to find any reference,
so that we next present a proof of this result for completeness.

\begin{proposition}\label{1641}
Let $\mathcal{M}$ be a von Neumann algebra with a normal semifinite faithful trace $\tau$ and with no minimal projections, 
let $X$ be a $q$-normed space, $0<q\le 1$, and 
let $\Phi\colon L^p(\mathcal{M},\tau)\to X$ be a continuous linear map with $0<p<q$.
Then $\Phi=0$.
\end{proposition}

\begin{proof}
The proof will be divided in a number of steps.

Our first step is to show that for each projection $e_0\in\mathcal{M}$ with $\tau(e_0)<\infty$ and each $0\le\rho\le\tau(e_0)$,
there exists a projection $e\in\mathcal{M}$ such that
$e\le e_0$ and $\tau(e)=\rho$.
Set 
\[
\mathcal{P}_1=\bigl\{e\in\mathcal{M} : e\text{ is a projection, } e\le e_0, \ \tau(e)\ge\rho\bigr\}.
\]
Note that $e_0\in\mathcal{P}_1$, so that $\mathcal{P}_1$ is non-empty.
Let $\mathcal{C}$ be a chain in $\mathcal{P}_1$, and let $e'=\wedge_{e\in\mathcal{C}}e$. 
Then $e'$ is a projection and $e'\le e_0$. For each $e\in\mathcal{C}$, since $\tau(e_0)<\infty$, it follows that
$\tau(e_0)-\tau(e)=\tau(e_0-e)$. From the normality of $\tau$ we now deduce that
\begin{align*}
\tau(e_0)-\inf_{e\in\mathcal{C}}\tau(e) & =
\sup_{e\in\mathcal{C}}\bigl(\tau(e_0)-\tau(e)\bigr) =
\sup_{e\in\mathcal{C}}\tau(e_0-e) \\ 
& = \tau\left(\vee_{e\in\mathcal{C}}(e_0-e)\right)=
\tau(e_0-e').
\end{align*}
Hence $\tau(e')=\inf_{e\in\mathcal{C}}\tau(e)\ge\rho$, which shows that $e'$ is a lower bound of $\mathcal{C}$,
and so, by Zorn's lemma, $\mathcal{P}_1$ has a minimal element, say $e_1$.
We now consider the set
\[
\mathcal{P}_2=\bigl\{e\in\mathcal{M} : e\text{ is a projection, } e\le e_1, \ \tau(e)\le\rho\bigr\}.
\]
Note that $0\in\mathcal{P}_2$, so that $\mathcal{P}_2$ is non-empty.
Let $\mathcal{C}$ be a chain in $\mathcal{P}_2$, and let $e'=\vee_{e\in\mathcal{C}}e$.
Then $e'\le e_1$, and the normality of $\tau$ yields
\[
\tau(e')=\sup_{e\in\mathcal{C}}\tau(e)\le\rho.
\]
This implies that $e'$ is an upper bound of $\mathcal{C}$, and so, 
by Zorn's lemma, $\mathcal{P}_2$ has a maximal element, say $e_2$.
Assume towards a contradiction that $e_1\ne e_2$.
Since, by hypothesis, $\mathcal{M}$ has  no minimal projections, it follows that there exists
a non-zero projection $e<e_1-e_2$. Since $e\perp e_2$, we see that $e_2+e$ is a projection.
Further,  we have $e_2<e_2+e<e_1$. 
The maximality of $e_2$ implies that $\tau(e_2+e)>\rho$, which implies that $e_2+e\in\mathcal{P}_1$,
contradicting the minimality of $e_1$.
Thus $e_1=e_2$, and this clearly implies that $\tau(e_1)=\tau(e_2)=\rho$.

Our next goal is to show that $\Phi(e_0)=0$ for each projection $e_0$ with $\tau(e_0)<\infty$.
From the previous step, it follows that there exists a projection $e\le e_0$ with $\tau(e)=\frac{1}{2}\tau(e_0)$.
Set $e'=e_0-e$. Then $\tau(e')=\frac{1}{2}\tau(e_0)$.
Further,
\[
\Vert\Phi(e_0)\Vert^q=
\Vert\Phi(e)+\Phi(e')\Vert^q\le
\Vert\Phi(e)\Vert^q+\Vert\Phi(e')\Vert^q,
\]
and therefore either $\Vert\Phi(e)\Vert^q\ge\frac{1}{2}\Vert\Phi(e_0)\Vert^q$ or 
$\Vert\Phi(e')\Vert^q\ge\frac{1}{2}\Vert\Phi(e_0)\Vert^q$. We define $e_1$ to be
any of the projections $e,e'$ for which the inequality holds. We thus get
$e_1\le e_0$, $\tau(e_1)=\frac{1}{2}\tau(e_0)$, and $\Vert\Phi(e_1)\Vert\ge 2^{-1/q}\Vert\Phi(e_0)\Vert$.
By repeating the process, we get a decreasing sequence of projections $(e_n)$ such that
\[
\tau(e_n)=2^{-n}\tau(e_0)\quad\text{and}\quad\Vert\Phi(e_n)\Vert\ge 2^{-n/q}\Vert\Phi(e_0)\Vert\quad(n\in\mathbb{N}).
\] 
Then
\[
\bigl\Vert 2^{n/q}e_n\bigr\Vert_p=2^{n/q}\tau(e_n)^{1/p}=2^{n(1/q-1/p)}\tau(e_0)^{1/p},
\]
which converges to zero, because $p<q$. Since $\Phi$ is continuous and
$\Vert\Phi(e_0)\Vert\le\bigl\Vert \Phi\bigl(2^{n/q}e_n\bigr)\bigr\Vert_p$ $(n\in\mathbb{N})$,
it may be concluded that $\Phi(e_0)=0$, as claimed.

Our next concern is to show that $\Phi$ vanishes on $S(\mathcal{M},\tau)$.
Of course, it suffices to show that $\Phi$ vanishes on $S(\mathcal{M},\tau)_+$.
Take $x\in S(\mathcal{M},\tau)_+$, and let $e=\soporte (x)$, so that $\tau(e)<\infty$.
The spectral decomposition implies that there exists a sequence
$(x_n)$ in $\mathcal{M}_+$ such that $\lim x_n=x$ with respect to
the operator norm and each $x_n$ is of the form
$x_n=\sum_{j=1}^k\rho_j e_j$,
where $\rho_1,\ldots,\rho_k\in\mathbb{R}^+$ and
$e_1,\ldots,e_k\in\mathcal{M}$ are mutually orthogonal projections
with $e_je=ee_j=e_j$ $(j\in\{1,\ldots,k\})$. From the previous step, we conclude that $\Phi(x_n)=0$ $(n\in\mathbb{N})$.
Further, from \eqref{1723} we deduce that
\[
\Vert x-x_n\Vert_p=\Vert e(x-x_n)\Vert_p\le\Vert e\Vert_p\Vert x-x_n\Vert\to 0,
\]
and the continuity of $\Phi$ implies that $\Phi(x)=0$, as required.

Finally, since  $S(\mathcal{M},\tau)$ is dense in $L^p(\mathcal{M},\tau)$ and $\Phi$ is continuous,
it may be concluded that $\Phi=0$.
\end{proof}

\begin{corollary}\label{1416}
Let $\mathcal{M}$ be a von Neumann algebra with a normal semifinite faithful trace $\tau$ and
with no minimal projections, 
let $X$ be a $q$-normed space, $0<q\le 1$, and
let $P\colon L^p(\mathcal{M},\tau)\to X$ be a continuous $m$-homogeneous polynomial
with $0<p/m<q$. Suppose that $P$ is orthogonally additive on $S(\mathcal{M},\tau)_+$.
Then $P=0$.
\end{corollary}

\begin{proof}
This is a straightforward consequence of Theorem~\ref{1527} and 
Proposition~\ref{1641}.
\end{proof}

We now turn our attention to the complex-valued polynomials. 
In this setting the representation given in Theorem~\ref{1527} 
has a particularly significant integral form, because of the well-known representation
of the dual of the $L^p$-spaces.
The trace gives rise to a distinguished contractive positive linear functional on $L^1(\mathcal{M},\tau)$, 
still denoted by $\tau$. 
By \eqref{1723}, if $\frac{1}{p}+\frac{1}{q}=1$, 
for each $\zeta\in L^q(\mathcal{M},\tau)$, the formula
\begin{equation}\label{2019}
\Phi_\zeta(x)=\tau(\zeta x)\quad (x\in L^p(\mathcal{M},\tau))
\end{equation}
defines a continuous linear functional on $L^p(\mathcal{M},\tau)$.
Further, in the case where $1\le p<\infty$,
the map $\zeta\mapsto\Phi_\zeta$ is an isometric isomorphism from $L^q(\mathcal{M},\tau)$
onto the dual space of $L^p(\mathcal{M},\tau)$.
It is immediate to see that $\Phi_\zeta^*=\Phi_{\zeta^*}$, so that $\Phi_\zeta$ is hermitian if and only
if $\zeta$ is self-adjoint.

\begin{corollary}\label{1415}
Let $\mathcal{M}$ be a von Neumann algebra with a normal semifinite faithful trace $\tau$, and 
let $P\colon L^p(\mathcal{M},\tau)\to \mathbb{C}$ be a continuous $m$-homogeneous polynomial
with $m\le p<\infty$.
Then the following conditions are equivalent:
\begin{enumerate}
\item[(i)]
there exists $\zeta\in L^{r}(\mathcal{M},\tau)$ 
such that $P(x)=\tau(\zeta x^m)$  $(x\in L^p(\mathcal{M},\tau))$,
where $r=p/(p-m)$ (with the convention that $p/0=\infty$);
\item[(ii)]
the polynomial $P$ is orthogonally additive on $L^p(\mathcal{M},\tau)_{\textup{sa}}$;
\item[(iii)]
the polynomial $P$ is orthogonally additive on $S(\mathcal{M},\tau)_+$.
\end{enumerate}
If the conditions are satisfied, then $\zeta$ 
is unique and $\Vert P\Vert\le\Vert\zeta\Vert_{r}\le 2\Vert P\Vert$;
moreover, if $P$ is hermitian, then $\zeta$ is self-adjoint and $\Vert\zeta\Vert_{r}=\Vert P\Vert$.
\end{corollary}

\begin{proof}
This follows from Theorems~\ref{1242} and~\ref{1527}.
\end{proof}

Let $H$ be a Hilbert space.
We denote by $\traza $ the usual trace on the von Neumann algebra $\mathcal{B}(H)$.
Then $L^p(\mathcal{B}(H),\traza )$, with $0<p<\infty$, is the Schatten class
$S^p(H)$. 
In the case where $0<p<q$, we have $S^p(H)\subset S^q(H)\subset\mathcal{K}(H)$ and 
$\Vert x\Vert\le\Vert x\Vert_q\le\Vert x\Vert_p$ $(x\in S^p(H))$.
It is clear that $S(\mathcal{B}(H),\traza )=\mathcal{F}(H)$, the two-sided ideal of $\mathcal{B}(H)$ 
consisting of the finite-rank operators.
Thus, the following result is an immediate consequence of Corollary~\ref{1415}.

\begin{corollary}
Let $H$ be a Hilbert space, and 
let $P\colon S^p(H)\to \mathbb{C}$ be a continuous $m$-homogeneous polynomial
with $m< p<\infty$. 
Then the following conditions are equivalent:
\begin{enumerate}
\item[(i)]
there exists $\zeta\in S^{r}(H)$ 
such that $P(x)=\traza (\zeta x^m)$  $(x\in S^p(H))$,
where $r=p/(p-m)$;
\item[(ii)]
the polynomial $P$ is orthogonally additive on $S^p(H)_{\textup{sa}}$;
\item[(iii)]
the polynomial $P$ is orthogonally additive on $\mathcal{F}(H)_+$.
\end{enumerate}
If the conditions are satisfied, then  $\zeta$ 
is unique and $\Vert P\Vert\le\Vert\zeta\Vert_{r}\le 2\Vert P\Vert$;
moreover, if $P$ is hermitian, then $\zeta$ is self-adjoint and $\Vert\zeta\Vert_{r}=\Vert P\Vert$.
\end{corollary}

\begin{corollary}
Let $H$ be a Hilbert space, and 
let $P\colon \mathcal{K}(H)\to \mathbb{C}$ be a continuous $m$-homogeneous polynomial. 
Then the following conditions are equivalent:
\begin{enumerate}
\item[(i)]
there exists $\zeta\in S^1(H)$ 
such that $P(x)=\traza (\zeta x^m)$  $(x\in\mathcal{K}(H))$;
\item[(ii)]
the polynomial $P$ is orthogonally additive on $\mathcal{K}(H)_{\textup{sa}}$;
\item[(iii)]
the polynomial $P$ is orthogonally additive on $\mathcal{F}(H)_+$.
\end{enumerate}
If the conditions are satisfied, then  $\zeta$ 
is unique and $\Vert P\Vert\le\Vert\zeta\Vert_{1}\le 2\Vert P\Vert$;
moreover, if $P$ is hermitian, then $\zeta$ is self-adjoint and $\Vert\zeta\Vert_{1}=\Vert P\Vert$.
\end{corollary}

\begin{proof}
In order to prove the equivalence of the conditions we are reduced to prove that (iii)$\Rightarrow$(i).
Suppose that (iii) holds.
Let $x,y\in\mathcal{K}(H)_+$ such that $x\perp y$. From the spectral decomposition of both $x$ and $y$
we deduce that there exist sequences $(x_n)$ and $(y_n)$ in $\mathcal{F}(H)_+$ such that
$\lim x_n=x$, $\lim y_n=y$, and $x_m\perp y_n$ $(m,n\in\mathbb{N})$. Then 
\[
P(x+y)=\lim P(x_n+y_n)=\lim \bigl(P(x_n)+P(y_n)\bigr)=P(x)+P(y).
\]
This shows that $P$ is orthogonally additive on $\mathcal{K}(H)_+$.
Since the $C^*$-algebra $\mathcal{K}(H)$ has real rank zero and 
the net consisting of all finite-rank projections is an increasing approximate unit,
Theorem~\ref{1137} applies and gives a continuous linear functional $\Phi$ on $\mathcal{K}(H)$ such 
that $P(x)=\Phi(x^m)$ $(x\in\mathcal{K}(H))$. It is well-known that the map $\zeta\mapsto\Phi_\zeta$,
as defined in \eqref{2019}, gives an isometric isomorphism from $S^1(H)$ onto the dual of $\mathcal{K}(H)$, 
so that there exists  $\zeta\in S^1(H)$ such that $\Phi(x)=\traza (\zeta x)$ $(x\in\mathcal{K}(H))$ and 
$\Vert\zeta\Vert_1=\Vert\Phi\Vert$. Thus we obtain (i).
The additional properties of the result follow from Theorem~\ref{1522}.
\end{proof}

\begin{corollary}\label{1418}
Let $H$ be a Hilbert space, and 
let $P\colon S^p(H)\to \mathbb{C}$ be a continuous $m$-homogeneous polynomial
with $0<p\le m$. 
Then the following conditions are equivalent:
\begin{enumerate}
\item[(i)]
there exists $\zeta\in\mathcal{B}(H)$ 
such that $P(x)=\traza (\zeta x^m)$  $(x\in S^p(H))$;
\item[(ii)]
the polynomial $P$ is orthogonally additive on $S^p(H)_{\textup{sa}}$;
\item[(iii)]
the polynomial $P$ is orthogonally additive on $\mathcal{F}(H)_+$.
\end{enumerate}
If the conditions are satisfied, then $\zeta$ 
is unique and $\Vert P\Vert\le\Vert\zeta\Vert\le 2\Vert P\Vert$;
moreover, if $P$ is hermitian, then $\zeta$ is self-adjoint and $\Vert\zeta\Vert=\Vert P\Vert$.
\end{corollary}

\begin{proof}
By Theorems \ref{1242} and \ref{1527},
it suffices to show that the map $\zeta\mapsto\Phi_\zeta$, as defined in \eqref{2019},
gives isometric isomorphism from $\mathcal{B}(H)$ onto the dual of $S^{p/m}(H)$. 
This is probably well-known, but we are not aware of any reference. Consequently, 
it may be helpful to include a proof of this fact.
If $\zeta\in\mathcal{B}(H)$ and $x\in S^{p/m}(H)$, then, by \eqref{1723}, $\zeta x\in S^{p/m}(H)$,
so that $\zeta x\in S^1(H)$ and
\[
\bigl\vert\traza (\zeta x)\bigr\vert\le
\Vert \zeta x\Vert_1\le\Vert\zeta\Vert\Vert x\Vert_{1}\le\Vert\zeta\Vert\Vert x\Vert_{p/m},
\]
which shows that $\Phi_\zeta$ is a continuous linear functional on $S^{p/m}(H)$ with $\Vert\Phi_\zeta\Vert\le\Vert\zeta\Vert$.
Conversely,
assume that $\Phi$ is a continuous linear functional on $S^{p/m}(H)$. 
For each $\xi,\eta\in H$, let $\xi\otimes\eta\in\mathcal{F}(H)$ defined by
\[
\bigl(\xi\otimes\eta\bigr)(\psi)=\langle\psi\vert\eta\rangle\xi\quad(\psi\in H),
\]
and define $\varphi\colon H\times H\to\mathbb{C}$ by
\[
\varphi(\xi,\eta)=\Phi(\xi\otimes\eta)\quad (\xi,\eta\in H).
\]
It is easily checked that $\varphi$ is a continuous sesquilinear functional with 
$\Vert\varphi\Vert\le\Vert\Phi\Vert$.
Therefore there exists $\zeta\in\mathcal{B}(H)$ such that 
$\langle\zeta(\xi)\vert\eta\rangle=\varphi(\xi,\eta)$ for all $\xi,\eta\in H$ and
$\Vert\zeta\Vert\le\Vert\Phi\Vert$.
The former condition implies that
\[
\Phi_\zeta(\xi\otimes\eta)=
\traza (\zeta\xi\otimes\eta)=
\langle\zeta(\xi)\vert\eta\rangle=
\varphi(\xi,\eta)=
\Phi(\xi\otimes\eta)
\] 
for all $\xi,\eta\in H$,
which gives $\Phi_\zeta(x)=\Phi(x)$ for each $x\in\mathcal{F}(H)$. Since $\mathcal{F}(H)$
is dense in $S^{p/m}(H)$, it follows that $\Phi_\zeta=\Phi$.
Further, we have $\Vert\zeta\Vert\le\Vert\Phi\Vert=\Vert\Phi_\zeta\Vert\le\Vert\zeta\Vert$.
Finally,
it is immediate to see that $\Phi_\zeta^*=\Phi_{\zeta^*}$, so that $\Phi_\zeta$ is hermitian if and only
if $\zeta$ is self-adjoint.
\end{proof}

\begin{proposition}\label{1526}
Let $H$ be a Hilbert space with $\dim H \geq 2$,
let $X$ be a topological linear space, and 
let $P\colon S^p(H) \to X$ be a continuous $m$-homogeneous polynomial with $0<p<\infty$.
Suppose that $P$ is orthogonally additive on $S^p(H)$. 
Then $P=0$.
\end{proposition}

\begin{proof}
Since $\mathcal{F}(H)$ is dense in $S^p(H)$ and $P$ is continuous,
it suffices to prove that $P$ vanishes on $\mathcal{F}(H)$.
On account of  Lemma~\ref{l1}, we are also reduced to prove that $P$ vanishes on $\mathcal{F}(H)_{\textup{sa}}$.
We continue to use the notation $\xi\otimes\eta$ which was introduced in the proof of Corollary~\ref{1418}.

Let $x\in\mathcal{F}(H)_{\textup{sa}}$. 
Then
$x=\sum_{j=1}^k\alpha_j\xi_j\otimes\xi_j$,
where $k\ge 2$, $\alpha_1,\ldots,\alpha_k\in\mathbb{R}$, and $\{\xi_1,\ldots,\xi_k\}$ is
an orthonormal subset of $H$. 
It is clear that the subalgebra $\mathcal{M}$ of $\mathcal{B}(H)$ generated by
$\bigl\{\xi_i\otimes\xi_j : i,j\in\{1,\ldots,k\}\bigr\}$ is contained in $\mathcal{F}(H)$ and 
it is $\ast$-isomorphic to the von Neumann algebra $\mathcal{B}(K)$, 
where $K$ is the linear span of the set $\{\xi_1,\ldots,\xi_k\}$. By 
Proposition~\ref{1525}, 
$P\mid_{\mathcal{M}}=0$,
and therefore $P(x)=0$. We thus get $P\mid_{\mathcal{F}(H)_{\textup{sa}}}=0$,  as required.
\end{proof}

\end{document}